\documentclass{amsart}[12pt]
\usepackage{amsmath, amsthm, amscd, amsfonts,}
\usepackage{graphicx,float}

\usepackage{amssymb}

\usepackage{amsopn}

\usepackage{comment}

\newcommand{\PP}{{\mathbb{P}}}
\newcommand{\QQ}{{\mathbb{Q}}}

\DeclareMathOperator{\dom}{dom}

\DeclareMathOperator{\supp}{supp}

\def\MPB{{\mathbb{P}}}
\def\MQB{{\mathbb{Q}}}

\def\k{\kappa}

\newtheorem{theorem}{Theorem}[section]
\newtheorem{lemma}[theorem]{Lemma}

\newtheorem{corollary}[theorem]{Corollary}

\newtheorem{definition}[theorem]{Definition}

\newtheorem{remark}[theorem]{Remark}

\newtheorem{claim}[theorem]{Claim}
\numberwithin{equation}{section}

\def\MPB{{\mathbb{P}}}
\def\MQB{{\mathbb{Q}}}

\def\k{\kappa}

\def\rmark{\mbox{$\rm\bf\rule{0.06em}{1.45ex}\kern-0.05em R$}}
\def\pmark{\mbox{$\rm\bf\rule{0.06em}{1.45ex}\kern-0.05em P$}}
\def\nmark{\mbox{$\rm\bf\rule{0.06em}{1.45ex}\kern-0.05em N$}}
\def\vdash{\mbox{$\rm\| \kern-0.13em -$}}
\newcommand{\lusim}[1]{\smash{\underset{\raisebox{1.2pt}[0cm][0cm]{$\sim$}}
{{#1}}}}

\usepackage{pb-diagram}

\def\rmark{\mbox{$\rm\bf\rule{0.06em}{1.45ex}\kern-0.05em R$}}
\def\pmark{\mbox{$\rm\bf\rule{0.06em}{1.45ex}\kern-0.05em P$}}
\def\nmark{\mbox{$\rm\bf\rule{0.06em}{1.45ex}\kern-0.05em N$}}
\def\vdash{\mbox{$\rm\| \kern-0.13em -$}}

\begin{document}

\title[On $C^s_n(\kappa)$ and the Juhasz-Kunen question]{On $C^s_n(\kappa)$ and the Juhasz-Kunen question}

\author[M. Golshani and S. Shelah]{Mohammad Golshani and Saharon Shelah}

\thanks{The first author's research has been supported by a grant from IPM (No. 96030417). The second
author's research has been partially supported by European Research Council grant 338821. Publication 1125 on Shelah's list.}
\thanks{The authors thank Ashutosh Kumar  for  many useful comments and corrections.}
 \maketitle




\begin{abstract}
We generalize the combinatorial principles  $C_n(\kappa), C^s_n(\kappa)$ and $Princ(\kappa)$ introduced by various authors, and prove some of their properties and connections between them. We also answer a question asked by Juhasz-Kunen about the relation between these principles, by showing that $C_n(\kappa)$ does not imply $C_{n+1}(\k)$, for any $n>2$. We also show  the consistency of $C(\kappa)+\neg C^s(\kappa).$
\end{abstract}

\section{introduction}
In this paper, we consider some combinatorial principles introduced in  \cite{brendle-fuchino}, \cite{fuchino-geschke}, \cite{juhasz-kunen}, \cite{juhasz-soukup1}, \cite{juhasz-soukup2} and \cite{shelah}, present some generalization of them and prove some of their properties and connections between them. We also answer a question asked by Juhasz-Kunen \cite{juhasz-kunen} about the relation between these principles.

The work in this direction, has started by the work of Juhasz-Soukup-Szentmiklossy \cite{juhasz-soukup1}, where the authors introduced several combinatorial principles, which all hold in the Cohen-real generic extensions. Among other things, in particular they introduced the combinatorial principles $C^s(\kappa)$, $C(\kappa),$ and their restrictions $C_m^s(\kappa)$ and $C_m(\kappa),$ for  $m<\omega.$ The work of Juhasz-Kunen \cite{juhasz-kunen} has continued the work, by introducing some extra principles, like $SEP$, and discussing their relations. In particular, Juhasz and Kunen showed that $SEP  \Rightarrow C^s_2(\aleph_2)$, while $C^s(\aleph_2) \nRightarrow SEP.$
On the other hand, in \cite{shelah}, Shelah introduced a new combinatorial principle $Princ(\kappa),$ which is weaker than $SEP$, but still enough strong to imply $C^s(\kappa).$

It turned out that these combinatorial principles are very useful, and have many applications, in particular in topology and the study of cardinal invariants, see \cite{brendle-fuchino}, \cite{fuchino-geschke} and \cite{juhasz-soukup2}.

The question of the difference between $C^s_n(\kappa)$ and $C^s_{n+1}(\kappa)$ remained open by Juhasz-Kunen \cite{juhasz-kunen} , and was also  asked  by Juhasz during the Beer-Sheva 2001 conference.

In this paper we consider some of these combinatorial principles, and present a natural generalization of them. We discuss the relation between them and also their consistency. We also address the above mentioned question of Juhasz-Kunen in section 5, and give a complete solution to it. In the last section, we discuss the relation between $C^s(\kappa)$ and $C(\kappa),$ and show the consistency of $C(\kappa)+\neg C^s(\kappa).$

\section{On $Princ(\kappa)$ and its generalizations}
In this section, we consider the combinatorial principle $Prin(\kappa)$ introduced by Shelah \cite{shelah}, and present some of its generalizations.
\begin{definition}
Let $\kappa$ be regular uncountable, $A \supseteq \kappa$, and let $D$ be a filter on $[A]^{<\kappa}.$ $D$ is called  normal  if
\begin{itemize}
\item [$(1)$] For all $a\in [A]^{<\kappa}, \{b\in [A]^{<\kappa}: a \subseteq b  \}\in D,$

\item [$(2)$] If for $x\in A, A_x\in D,$ then $\bigtriangleup_{x\in A}A_x\in D,$ where
\begin{center}
$\bigtriangleup_{x\in A}A_x=\{a\in [A]^{<\kappa}: \forall x\in a, a\in A_x \}.$
\end{center}
\end{itemize}
\end{definition}
It is easily seen that if $D$ is a normal filter on $[A]^{<\kappa}, X\neq \emptyset$ mod $D$ and if $F:X \rightarrow A$ is regressive, i.e.,  for all non-empty $a\in [A]^{<\kappa}, F(a)\in a,$ then there are $Y \subseteq X, Y \neq \emptyset$ mod $D$ and  $x\in A$ such that for all $a\in Y, F(a)=x.$ To see this, assume on the contrary that for each $x\in A,$ there exists $Y_x \in D$ such that $Y_x \cap \{a\in X: F(a)=x    \}=\emptyset.$ Let $Y=\bigtriangleup_{x\in A}Y_x.$ Then $Y\in D$ and so $Y\cap X \neq \emptyset$ (as $X\neq \emptyset$ mod $D$). Let $a\in Y\cap X$ and $F(a)=x.$ Then $a\in Y_x \cap \{a\in X: F(a)=x    \},$ a contradiction.
\begin{definition}
Let $\kappa$ be regular uncountable. $\mathfrak{D}$ is a $\kappa$-definition of normal filters, if
\begin{itemize}
\item [$(1)$] For each $A \supseteq \kappa, \mathfrak{D}(A)$ is a normal filter on $[A]^{<\kappa},$

\item [$(2)$] If $\kappa \subseteq A_1 \subseteq A_2,$ then $\mathfrak{D}(A_1)=\{ \{a\cap A_1: a\in X \}: X\in \mathfrak{D}(A_2) \}.$
\end{itemize}
\end{definition}
\begin{definition}
Let $\kappa$ be regular uncountable, $\theta< \lambda \leq \kappa$ and let  $\chi > \kappa$ be large enough regular. Then
\begin{itemize}
\item [$(a)$] $\mathbf{N}^1_{\kappa, \lambda, \chi }$ consists of those $N \prec (H(\chi), \in)$ such that:
\begin{enumerate}
\item $|N|\leq N\cap \kappa\in \kappa,$
\item For all $a\in P(\omega),$ there exists $P\in N,$ such that $P \subseteq P(\omega), |P|< \min\{|N|^+, \lambda\},$ and for all $b\in P(\omega)\cap N, a \subseteq b \Rightarrow \exists c\in P, a \subseteq c \subseteq b$ (such a $P$ is called an $N$-witness for $a$).
\end{enumerate}
\item [$(b)$] $\mathbf{N}^2_{\kappa, \lambda, \theta,  \chi }$ consists of those $N \in \mathbf{N}^1_{\kappa, \lambda, \chi }$ such that for any $\theta$-sequence $\langle a_\xi: \xi<\theta \rangle$ of subsets of $\omega,$ there is some $P\in N,$ $P \subseteq P(\omega), |P|< \min\{|N|^+, \lambda\},$ such that $P$ is an $N$-witness for all $a_\xi, \xi<\theta$ simultaneously.

\item [$(c)$] $\mathbf{N}^3_{\kappa, \lambda, \theta,  \chi }$ consists of those $N \in \mathbf{N}^1_{\kappa, \lambda, \chi }$ such that for each $Y\in [N]^{\theta},$ there exists some $Z\in N, |Z|< \min\{|N|^+, \lambda  \}$ such that $Y \subseteq Z.$
\end{itemize}
\end{definition}
We now state our generalization of $Princ(\kappa)$
\begin{definition}
Let $\theta<\lambda\leq \kappa$ and $\mathfrak{D}$ be as above.
\begin{itemize}
\item [$(a)$] $Princ_1(\kappa, \lambda, \mathfrak{D})$ sates: for all large enough $\chi> \kappa,$
\begin{center}
$\mathbf{N}^1_{\kappa, \lambda, \chi } \neq \emptyset$ mod $\mathfrak{D}(H(\chi))$.
\end{center}

\item [$(b)$] $Princ_{l,\theta}(\kappa, \lambda, \mathfrak{D})$ (for $l=2,3$) sates: for all large enough $\chi> \kappa,$
\begin{center}
$\mathbf{N}^l_{\kappa, \lambda, \theta, \chi } \neq \emptyset$ mod $\mathfrak{D}(H(\chi))$.
\end{center}
\end{itemize}
\end{definition}
\begin{remark}
\begin{itemize}
\item [$(a)$] Let $\kappa$ be regular uncountable, and for $A \supseteq \kappa$ let $\mathfrak{D}(A)$ be the club filter on $[A]^{<\kappa}.$ Then our $Princ_1(\kappa, \lambda, \mathfrak{D})$ is just $Princ(\kappa, \lambda)$ from \cite{brendle-fuchino}. Also note that Shelah's $Princ(\kappa)$ is $Princ(\kappa, \kappa).$

\item [$(b)$] If $\theta<\lambda \leq \lambda' \leq \kappa,$ then $Princ_1(\kappa, \lambda, \mathfrak{D}) \Rightarrow Princ_1(\kappa, \lambda', \mathfrak{D})$ and
$Princ_{l,\theta}(\kappa, \lambda, \mathfrak{D}) \Rightarrow$

$\hspace{0.5cm}$$Princ_{l,\theta}(\kappa, \lambda',  \mathfrak{D})$ (for $l=2,3$).

\item [$(c)$] If $\theta \leq \theta' < \lambda \leq \kappa,$ then
$Princ_{l,\theta'}(\kappa, \lambda, \mathfrak{D}) \Rightarrow Princ_{l,\theta}(\kappa, \lambda, \mathfrak{D})$ (for $l=2,3$).

\item [$(d)$] If $\lambda=\mu^+$ is a successor cardinal, then we can replace $\min\{|N|^+, \lambda  \}$ by $\lambda.$
\end{itemize}
\end{remark}
The next lemma follows from the definition, and the fact that we can code an $\omega$-sequence of subsets of $\omega$ into a subset of $\omega$.
\begin{lemma}
Let $\theta<\lambda\leq \kappa$ and $\mathfrak{D}$ be as above. Then
\begin{itemize}
\item [$(a)$] $Princ_{3,\theta}(\kappa, \lambda, \mathfrak{D}) \Rightarrow Princ_{2,\theta}(\kappa, \lambda, \mathfrak{D}) \Rightarrow Princ_1(\kappa, \lambda, \mathfrak{D}).$

\item [$(b)$] $Princ_{2,\omega}(\kappa, \lambda, \mathfrak{D}) \Leftrightarrow Princ_1(\kappa, \lambda, \mathfrak{D}).$
\end{itemize}
\end{lemma}
\begin{proof}
$(a)$ is by definition, let's prove $(b)$. It suffices to show that
\[
\mathbf{N}^2_{\kappa, \lambda, \omega,  \chi } = \mathbf{N}^1_{\kappa, \lambda, \chi }.
\]
Let $\Gamma: \omega \times \omega \rightarrow \omega$ be the Godel pairing function.  Let $N \in \mathbf{N}^1_{\kappa, \lambda, \chi },$
 and suppose that $\langle a_n: n<\omega \rangle$ is a sequence of subsets of $\omega.$ Let
 \[
 a^*=\{ \Gamma(i, n): n<\omega, i \in a_n    \}.
 \]
Let $P^* \in N$ be an $N$-witness for $a^*.$ Let
\[
P=\{ \{i: \Gamma(i, n) \in b   \}: n<\omega, b \in P^*                        \}.
\]
We show that $P$ is an $N$-witness for all $a_n, n<\omega,$ simultaneously. Clearly $P\in N, P \subseteq P(\omega) \cap N$ and $|P| < \min\{|N|^+, \lambda  \}$. Now let $n< \omega, ~ b \in P(\omega) \cap N$ and assume $a_n \subseteq b.$ Let
$b^{[n]}=\{ \Gamma(i, m): i, m < \omega$ and $m=n \Rightarrow i\in b        \}.$
Clearly $b^{[n]}  \in P(\omega) \cap N$ and $a^* \subseteq b^{[n]}.$ Hence by the choice of $P^*,$ there is $c \in P^*$ such that $a^* \subseteq c \subseteq b^{[n]}.$ Let $c^{[n]}=\{ i: \Gamma(i, n)\in c   \}$. Then $c^{[n]} \in P,$ and we can easily see that $a_n \subseteq c^{[n]} \subseteq b.$
We are done.
\end{proof}

\section{On $C^s(\kappa)$ and its generalizations}
Recall that for a filter $D$ on a set $I$,  $D^+$  is defined by
\begin{center}
$D^+=\{X \subseteq I: I\setminus X \notin D \}$.
\end{center}
It is clear that $D \subseteq D^+.$
\begin{definition}
Suppose $\kappa$ is regular uncountable, $D$ is a filter on $\kappa, J$ is an ideal on $\omega$ and $T$ is a subtree of $\theta^{<\omega}.$
\begin{itemize}
\item [$(a)$] The combinatorial principle
$C^D_T(\kappa, J)$ states: for any $(\kappa\times\theta)$-matrix $\bar{A}=\langle a_{\alpha, \xi}: \alpha<\kappa, \xi<\theta   \rangle$ of subsets of $\omega,$ one of the following holds:
\begin{itemize}
\item [$(\alpha)$]: There exists $S\in D^+$ such that for all $t\in T\cap \theta^n$ and all distinct $\alpha_0, \dots, \alpha_{n-1}\in$
$S, \bigcap_{i<n}a_{\alpha_i, t(i)} \neq \emptyset$ mod $J$.

\item [$(\beta)$]: There are $t\in T\cap \theta^n,$ for some $0<n<\omega,$ and $S_0, \dots, S_{n-1}\in D^+$ such that for
 all distinct $\alpha_i\in S_i, i<n$ we have $\bigcap_{i<n}a_{\alpha_i, t(i)} = \emptyset$ mod $J$.
\end{itemize}
\item [$(b)$] $C^D(\kappa, J)$ is $C^D_{T}(\kappa, J)$ for all trees $T \subseteq \theta^{<\omega}.$

\item [$(c)$] For $m<\omega,$ the combinatorial principles
$C^D_{T,m}(\kappa, J)$ and $C^D_{m}(\kappa, J)$ are defined similarly, where we require $T \subseteq \theta^{\leq m}.$
\end{itemize}
\end{definition}
\begin{remark} Suppose that $\kappa$ is regular uncountable and $m<\omega$.
\begin{itemize}
\item [$(a)$] If $D$ is the club filter on $\kappa,$ and $J=\{\emptyset\}$, then $C^D(\kappa, J), C^D_{m}(\kappa, J)$ are respectively the principles
 $C^s(\kappa), C^s_m(\kappa)$  from \cite{juhasz-soukup1}.

\item [$(b)$] If $D$ is the filter of co-bounded subsets of $\kappa,$ and $J=\{\emptyset\}$, then $C^D(\kappa, J), C^D_{m}(\kappa, J)$
 are respectively the principles $C(\kappa), C_m(\kappa)$ from \cite{juhasz-soukup1}.
\end{itemize}
\end{remark}
\begin{theorem}
Assume $\theta< \kappa \leq 2^{\aleph_0}, \kappa$ is regular, $J$ is an ideal on $\omega, T$ is a subtree of $\theta^{<\omega},$ and suppose that $Prin_{2,\theta}(\kappa, \kappa, \mathfrak{D})$ holds, where $\mathfrak{D}$ is a definition of $\kappa$-normal filters. Then $C^D_T(\kappa, J)$ holds, where $D$ is any filter on $\kappa$ satisfying: for $X\in D$ and $N\in \mathbf{N}^2_{\kappa, \lambda, \theta,\mathfrak{D}, \chi }$ with $D\in N$, $X\in N \Rightarrow \delta(N)= N \cap \kappa\in X.$
\end{theorem}
\begin{remark}
If $D$ is the club filter on $\kappa$ or the filter of co-bounded subsets of $\kappa,$ then $D$ has the above mentioned property.
\end{remark}
\begin{proof}
Let $\bar{A}=\langle a_{\alpha, \xi}: \alpha<\kappa, \xi<\theta   \rangle$ be a $(\kappa\times\theta)$-matrix of subsets of $\omega.$
 Let $\chi >  2^{\aleph_0} $ be large enough regular. By our assumption
\begin{center}
$\mathbf{N}^2_{\kappa, \kappa, \theta, \chi } \neq \emptyset$ mod $(\mathfrak{D}(H(\chi)))^+$.
\end{center}
Hence by normality of the filter,
\begin{center}
$\mathcal{N}=\{ N\in \mathbf{N}^2_{\kappa, \kappa, \theta, \chi }: D, \bar{A}\in N   \}\in \mathfrak{D}(H(\chi))^+.$
\end{center}
For $N\in \mathcal{N},$ set $\delta(N)=N\cap \kappa \in \kappa.$ By our assumption, for each $N\in \mathcal{N},$ we can find $P_N\in N$ such that $P_N$ is an $N$-witness for each $a_{\delta(N), \xi}, \xi<\theta,$ simultaneously. Then the map
$N \mapsto P_N$
is regressive on $\mathcal{N},$ so
by the normality of the filter $\mathfrak{D}(H(\chi))$, we can find $\mathcal{N_*} \subseteq \mathcal{N}$ and $P_*$ such that
 $\mathcal{N_*}\in \mathfrak{D}(H(\chi))^+,$ and
 for all $N\in \mathcal{N_*}, P_N=P_*.$
Let
\begin{center}
$S=\{\delta(N): N\in \mathcal{N_*}\}.$
\end{center}
\begin{claim}
$S\in D^+.$
\end{claim}
\begin{proof}
Suppose not; so $\kappa\setminus S \in D.$ But then for all $N\in \mathcal{N_*},$
\begin{center}
$\kappa\setminus S \in N \Rightarrow \delta(N)\in \kappa\setminus S.$
\end{center}
On the other hand by normality of the filter $\mathfrak{D}(H(\chi))$, we have
\begin{center}
$\mathcal{N_{**}}=\{N\in \mathcal{N_*}: \kappa\setminus S\in N \} \in \mathfrak{D}(H(\chi))^+,$
\end{center}
in particular $\mathcal{N_{**}} \neq \emptyset$. Let $N\in \mathcal{N_{**}}.$ Then we have
 $\kappa\setminus S \in D$, which implies $\delta(N)\in \kappa\setminus S.$ But on the other hand
$ N\in \mathcal{N_{*}}$ (as $\mathcal{N_{**}} \subseteq \mathcal{N_{*}}$), which implies $ \delta(N)\in S,$
 a contradiction.
\end{proof}
If for all $t\in T\cap \theta^n$ and all distinct $\alpha_0, \dots, \alpha_{n-1}\in S,$ we have $\bigcap_{i<n}a_{\alpha_i, t(i)} \neq \emptyset$ mod $J,$ then case $(\alpha)$ of Definition 3.1$(a)$ holds and we are done. Otherwise, we can find $t\in T\cap \theta^n$ and distinct $\alpha_0, \dots, \alpha_{n-1}\in S,$ such that $\bigcap_{i<n}a_{\alpha_i, t(i)} = \emptyset$ mod $J.$

For each $i<n,$ set $N_i\in \mathcal{N_*}$ be such that $\alpha_i=\delta(N_i).$ We also assume w.l.o.g. that $\alpha_0 < \dots < \alpha_{n-1}.$
\begin{claim}
There are $c_0, \dots, c_{n-1}\in P_*$ such that:
\begin{enumerate}
\item $\bigcap_{i<n}c_i =\emptyset$ mod $J$,
\item $i< n \Rightarrow a_{\alpha_i, t(i)} \subseteq c_i$ mod $J$.
\end{enumerate}
\end{claim}
\begin{proof}
We construct the sets $c_i, i<n,$ by downward induction on $i,$ so that for all $i<n,$

$(*)_i$ $\hspace{3.cm}$ $\bigcap_{j<i}a_{\alpha_j, t(j)} \cap \bigcap_{i\leq j< n}c_j =\emptyset$ mod $J$.

For $i=n,$ there is nothing to prove; thus suppose that $i<n$ and $c_{i+1}\in P_*$ is defined, so that $(*)_{i+1}$ is satisfied. It then follows that
\begin{center}
$a_{\alpha_i, t(i)} \subseteq b_i= \omega \setminus (\bigcap_{j<i}a_{\alpha_j, t(j)} \cap \bigcap_{i+1 \leq j <n}c_j)$ mod $J$.
\end{center}
It is easily seen that $b_i\in N_i,$ so as $P_*$ is an $N_i$-witness for $a_{\alpha_i, t(i)},$ we can find $c_i\in P_*$ so that
\begin{center}
$a_{\alpha_i, t(i)} \subseteq c_i \subseteq b_i$ mod $J$.
\end{center}
It is easily seen that $c_0, \dots, c_{n-1}$ are as required.
\end{proof}
For $i<n,$ set
\begin{center}
$S_i=\{ \alpha\in \kappa: a_{\alpha, t(i)} \subseteq c_i$ mod $J   \} \in N_i.$
\end{center}
\begin{claim}
For each $i<n, S_i\in D^+.$
\end{claim}
\begin{proof}
Suppose not; then $\kappa\setminus S_i\in D.$ But as $\kappa\setminus S_i \in N_i,$ we have $\alpha_i=\delta(N_i)\in \kappa\setminus S_i,$ which is a contradiction.
\end{proof}
Now if $\beta_i\in S_i$ are distinct, then
\begin{center}
$\bigcap_{i<n}a_{\beta_i, t(i)} \subseteq \bigcap_{i<n}c_i =\emptyset$ mod $J$,
\end{center}
and hence case $(\beta)$ of Definition 3.1$(a)$ holds and we are done.
The theorem follows.
\end{proof}
\begin{corollary}
 Assume $\kappa \leq 2^{\aleph_0}$ is regular uncountable. Then $Princ(\kappa)$ implies $C^s(\kappa).$
\end{corollary}

\section{Forcing $Princ_1(\kappa, \kappa, \mathfrak{D})$}
In this section we consider the  principles $Princ_1(\kappa, \lambda, \mathfrak{D})$ and $Princ_{2,\theta}(\kappa, \lambda, \mathfrak{D}),$ where $\theta<\lambda\leq \kappa=cf(\kappa)$ and $\mathfrak{D}$ is a $\kappa$-definition of normal filters, and discuss their consistency.
In fact, we will show that  in the generic extension by the Cohen forcing   $Add(\omega, \kappa)$ the above principles hold. We prove the result
for $Princ(\kappa)$, as the other cases can be proved similarly.

Recall that the Cohen forcing $Add(\omega, I)$ for adding $|I|$-many new Cohen subsets of $\omega$ is defined as
\begin{center}
$Add(\omega, I)=\{p: \omega \times I \to 2: |p| < \aleph_0           \}$,
\end{center}
ordered by reverse inclusion.

For a nice name $\lusim{a} = \bigcup_{n<\omega} \{\check{n}\} \times A_n$, where each $A_n$ is  a maximal antichain in $Add(\omega, \lambda),$ set
\[
\supp(\lusim{a})=\{\alpha \in \lambda: \exists n<\omega, \exists p \in A_n, \exists k \in \omega, (k, \alpha) \in \dom(p)     \}.
\]
Note that, by the countable chain condition property of $Add(\omega, \lambda),$ $\supp(\lusim{a})$ is a countable set, and $\lusim{a}$ can be considered as an $Add(\omega, \supp(\lusim{a}))$-name.
The following lemma follows easily by an absoluteness argument.
\begin{lemma}
Assume $U \subseteq \lambda$, $\lusim{a}_1, \dots, \lusim{a}_n$ are $Add(\omega, U)$-names,   $\phi(v_1, \dots, v_n)$ is a $\Delta^{\text{ZFC}}_1$-formula
and $p \in Add(\omega, \lambda)$. Then
\begin{center}
$p \Vdash_{Add(\omega, \lambda)}$``$\phi(\lusim{a}_1, \dots, \lusim{a}_n)$'' $\iff$ $p\upharpoonright \omega \times U \Vdash_{Add(\omega, U)}$``$\phi(\lusim{a}_1, \dots, \lusim{a}_n)$''.
\end{center}
\end{lemma}
We also need the following simple observation.
\begin{lemma}
Let $\mathfrak{D}$ be defined by $\mathfrak{D}(A)=$the club filter on $[A]^{<\kappa}$. The following are equivalent:
\begin{itemize}
\item [$(a)$] $Princ(\kappa)$.
\item [$(b)$] For all large enough $\chi > \kappa$ and $x \in H(\chi)$, there exists $N \in  N^1_{\kappa, \kappa,  \chi}$ such that $x \in N.$
\end{itemize}
\end{lemma}
\begin{proof}
It is clear that $(a) \implies (b)$. To show that $(b)$ implies $(a)$, let $\chi > \kappa$ be large enough regular, $x \in H(\chi)$ and let
$C \subseteq [H(\chi)]^{< \kappa}$ be a club set. We need to show that $N^1_{\kappa, \kappa,  \chi} \cap C \neq \emptyset.$ We assume $|M| \leq M\cap \kappa$
for all $M \in C.$ Take $\chi' > \chi$ large enough regular. By the assumption, we can find $N' \in N^1_{\kappa, \kappa,  \chi'}$
such that $x, C \in N'.$ Let $N= N' \cap H(\chi)$. Then by elementarity, $N \in N^1_{\kappa, \kappa,  \chi}$
and $N= \bigcup (N' \cap C).$ Since $C$ is closed, $N \in C$ and so $N \in N^1_{\kappa, \kappa,  \chi} \cap C,$
as required.
\end{proof}

We are now ready to show that  $Princ(\kappa)$ holds in the generic extension by Cohen forcing. We follow the proof in \cite{fuchino}.
\begin{theorem}
\label{forcign with add special case}
Assume $\lambda \geq \kappa= cf(\kappa) > 2^{\aleph_0}$.
Then $\Vdash_{Add(\omega, \lambda)}$``$Princ(\kappa)$''.
\end{theorem}

\begin{proof}
Let $\chi  > \lambda$ and $p \in Add(\omega, \lambda)$ be such that
\begin{center}
$p \Vdash_{Add(\omega, \lambda)}$``$\lusim{X}$ has transitive closure of cardinality $< \chi$''.
\end{center}
Let $\langle N_i : i < \delta     \rangle$ be a sequence of elementary submodels of $(H(\chi), \in)$ such that:
\begin{enumerate}
\item $\delta < \kappa, cf(\delta) > \aleph_0,$
\item  $\langle N_i : i < \delta     \rangle$  is increasing continuous,
\item $N_i \cap \kappa \in \kappa,$
\item  $\langle N_j : j \leq i     \rangle \in N_{i+1},$
\item Each $N_{i+1}, i < \delta,$ is closed under countable sequences,
\item $|N_i| < \kappa,$
\item $p, \lambda, \lusim{X} \in N_0.$
\end{enumerate}
Let $N=\bigcup_{i<\delta}N_i.$ As $cf(\delta)> \aleph_0$, it follows from clause $(5)$ that $N$ is closed under countable sequences.

We show that $p \Vdash$`` $\bigcup_{i<\delta}N_i[\dot{G}] \in N^1_{\kappa, \kappa, \chi}$''.
Let  $G$ be $Add(\omega, \lambda)$-generic over $V$ with $p \in G$,  $N^*_i=N_i[G], i < \delta,$ and $N^*= \bigcup_{i<\delta}N^*_i.$ Note that $N^*$ is closed under countable sequences and $\lusim{x}[G] \in N^*$. We show that $N^* \in (N^1_{\kappa, \kappa,  \chi})^{V[G]}$. This will complete the proof by the previous lemma.

Thus assume that $a \in P(\omega)^{V[G]}$ and let $\lusim{a}$ be a nice name for $a$. Let $U=\supp(\lusim{a})$.

Set $U_1=U \cap N$ and $U_2=U \setminus N.$ Then $U_1 \in N.$
Let $i<\delta$ be sufficiently large such that $U_1 \subseteq N_i$ and $|N \cap \lambda|=|N_i \cap \lambda|$\footnote{Such an $i$ exists as $U=\supp(\lusim{a})$ is a countable set
and $cf(\delta)> \aleph_0$.} and set $M= N_{i+1}.$ It follows from $(5)$ that $U_1 \in M.$

Let $\pi: \lambda \simeq \lambda$ be a bijection such that  $\pi[U] \subseteq M$, $\pi \upharpoonright U_1 = id \upharpoonright U_1$ and $\pi[\lambda \cap N] \subseteq M$. Using the homogeneity of the forcing $Add(\omega, \lambda)$,
  extend $\pi$
to an isomorphism $\pi: Add(\omega, \lambda) \simeq Add(\omega, \lambda)$. Note that this also induces an isomorphism of the class of all $Add(\omega, \lambda)$-names, $V^{Add(\omega, \lambda)}$, that we still denote it  by $\pi.$
Note that $\pi[U], \pi(\lusim{a}) \in M$, as $M$ is closed under countable sequences.
Let
\begin{center}
$P= \{ \lusim{c}_r[G]: r \in Add(\omega, \lambda \cap M \setminus U) \},$
\end{center}
where for $r \in Add(\omega, \lambda \cap M \setminus U)$,
\begin{center}
$\Vdash_{\MPB}$``$\lusim{c}_r = \omega \setminus \{n \in \omega: \exists q \in \dot{G} \cap Add(\omega, U_1),~ r \cup q \Vdash$``$n \notin \pi(\lusim{a})$''$            \}$''.
\end{center}
Note that $\lambda \cap M \setminus U = \lambda \cap M \setminus U \cap M \in N,$ so $ Add(\omega, \lambda \cap M \setminus U) \in N.$ Also $U_1 \in N$
so $Add(\omega, U_1) \in N.$ It easily follows that
 $P \in N^*$. It is also clear that $P \subseteq P(\omega)^{V[G]}$ and $|P| \leq |Add(\omega, \lambda \cap M \setminus U)| = |M| < \kappa.$

 To show that $P$ is an $N^*$-witness for $a$,
let $b \in P(\omega)^{V[G]} \cap N^*$ and $b \supseteq a.$ Let $\lusim{b} \in N$ be a nice name for $b$ and let $W=\supp(\lusim{b}).$  Let $p^* \leq p$ be such that $p^* \Vdash$``$\lusim{b} \supseteq \lusim{a}$''.
By Lemma 4.1, we may suppose that
$p^* \in Add(\omega, U \cup W).$
Let
\begin{center}
 $r=\pi(p^* \upharpoonright \omega \times (\lambda  \setminus U_1))$.
 \end{center}
 Then $r \in Add(\omega, \lambda \cap M \setminus U).$ We complete the proof by showing that
 $a \subseteq \lusim{c}_r[G] \subseteq b$.
 \\
{\bf \underline{$a \subseteq \lusim{c}_r[G]$} :} Assume by contradiction that $n \in a \setminus \lusim{c}_r[G]$. Let $q \in G, q \leq p^*$ and
\[
q \Vdash\text{~``~}n \in \lusim{a}\text{~''~}.
\]
Again we can suppose that $q \in Add(\omega, U \cup W).$ As $n \notin \lusim{c}_r[G],$ we can find $q^* \in G \cap Add(\omega, U_1)$ such that
\[
r \cup q^* \Vdash\text{~``~} n \notin \pi(\lusim{a})   \text{~''.~}
\]
This implies
\[
\pi^{-1}(r) \cup \pi^{-1}(q^*) \Vdash\text{~``~} n \notin \lusim{a}   \text{~''.~}
\]
Note that $\pi^{-1}(r) \cup \pi^{-1}(q^*)=p^* \upharpoonright \omega \times (\lambda  \setminus U_1) \cup q^*$. As $q^*, q \upharpoonright \omega \times U_1 \in G \cap Add(\omega, U_1),$ they are compatible, and we can easily conclude that $q$ and $p^* \upharpoonright \omega \times (\lambda  \setminus U_1) \cup q^*$
are compatible, which is a contradiction, as they decide the statement ``$n \in \lusim{a}$'' in different ways.
\\
{\bf \underline{$\lusim{c}_r[G] \subseteq b$} :} Suppose by contradiction that there is some $n \in \lusim{c}_r[G] \setminus b$.
Let $q \in G, q \leq p^*,$ be such that $q \Vdash$``$n \notin \lusim{b}$''. We can suppose that $q \in Add(\omega, U \cup W)$
and $q \restriction U_2 = p^* \restriction U_2.$
As  $p^* \Vdash$``$\lusim{b} \supseteq \lusim{a}$'', we have  $q \Vdash$``$n \notin \lusim{a}$'',
and hence
  $q \upharpoonright U \Vdash$``$n \notin \lusim{a}$''. Applying $\pi,$ we have
\begin{center}
$\pi(q \upharpoonright U) \Vdash$``$n \notin \pi(\lusim{a})$''.
\end{center}
Hence
\begin{center}
$q \restriction U_1 \cup \pi(q \restriction U_2) \Vdash$``$n \notin \pi(\lusim{a})$'',
\end{center}
which implies
\begin{center}
$q \restriction U_1 \cup \pi(p^* \restriction U_2) \Vdash$``$n \notin \pi(\lusim{a})$'',
\end{center}
Now observe that $r \leq \pi(p^* \restriction U_2)$ and $r$ is compatible with $q \restriction U_1$, so
\begin{center}
$r \cup q \restriction U_1 \Vdash$``$n \notin \pi(\lusim{a})$''.
\end{center}
Thus, $r \cup q \upharpoonright U_1$  witnesses  $n \notin \lusim{c}_r[G]$,
which is a contradiction.

The theorem follows.
\end{proof}
The next theorem can be proved as in Theorem \ref{forcign with add special case}.
\begin{theorem}
Assume $\theta<\lambda\leq \kappa=cf(\kappa)$  and $\mathfrak{D}$ is a $\kappa$-definition of normal filters. Then $\Vdash_{Add(\omega, \kappa)}$``$Princ_1(\kappa, \lambda, \mathfrak{D})+Princ_{2,\theta}(\kappa, \lambda, \mathfrak{D})$''. The same result holds with $\mathfrak{D}^+$ replaced with $\mathfrak{D}$.
\end{theorem}
\begin{remark}
In $V[G], \mathfrak{D}$ is defined as follows: For any large enough $\chi > \kappa, \mathfrak{D}(H^{V[G]}(\chi))$ is the filter generated by $\{     N[G]: N\in X\},$ where $X\in \mathfrak{D}(H(\chi)).$ Note that $N \prec H(\chi) \Rightarrow N[G] \prec H(\chi)[G]=H(\chi)^{V[G]},$ and so the above definition is well-defined.
\end{remark}

\section{On a question of Juhasz- Kunen}
In this section we answer a question of Juhasz- Kunen \cite{juhasz-kunen}, by showing that for $n\geq 2,$ $C_n(\aleph_2) \nRightarrow C_{n+1}(\aleph_2).$ In fact we prove the following  stronger result.
\begin{remark}
The results of this section are stated and proved for the ideal $J=[\omega]^{<\omega}$; but all of them are valid if we also assume $J=\{\emptyset\}.$
\end{remark}
\begin{theorem}
Assume:

$(1)$ $\aleph_0 < \theta=\theta^{<\theta} < \kappa = cf(\kappa)< \chi$ and $\forall \alpha < \chi (|\alpha|^{<\theta} < \chi),$

$(2)$ $D$ is a $\kappa$-complete filter on $\kappa$ which satisfies the $\Delta$-system $\theta$-property (see below for

$\hspace{0.5cm}$the definition) and contains the co-bounded subsets of $\kappa$, and $J=[\omega]^{<\omega}$,

$(3)$ $2 < n(*) < \omega.$

Then there is a cofinality preserving generic extension of the universe in which  $C^D_n(\kappa, J)$ holds if $n<n(*),$
and   fails if $n=n(*).$\footnote{When working in a forcing extension, we use $D$ to denote  the filter generated by $D$ in that extension.}
\end{theorem}
The rest of this section is devoted to the proof of the above theorem. The forcing notion we define is of the form $\PP_\chi * \lusim{\QQ}_{\kappa, \mathcal{A}},$ where $\PP_\chi$ is a suitable iteration of length $\chi,$ which adds a set $\mathcal{A} \subseteq [\kappa]^{< \aleph_0},$ which has nice enough properties. Then we use this added set $\mathcal{A}$ to define the forcing notion $\QQ_{\kappa, \mathcal{A}}.$

In subsection 5.1 we define the notion of having the $\Delta$-system $\theta$-property for a filter $D$, and show that under suitable conditions, some filters have this property. In subsection 5.2 we define the forcing notion $\PP_\chi$ and prove its basic properties. Subsection 5.3 is devoted to the definition of the forcing notion $\QQ_{\kappa, \mathcal{A}}.$ Finally in subsection 5.4 we complete the proof of the above theorem.
\subsection{Filters with the $\Delta$-system $\theta$-property}
In this subsection we prove a generalized version of $\Delta$-system lemma that will be used several times later.
\begin{definition}
Let $D$ be a filter on $\kappa,$ and $\theta < \kappa$ be a cardinal. $D$ has the $\Delta$-system $\theta$-property, if for any $Y \subseteq \kappa, Y\neq \emptyset$ mod $D$, and any sequence $\langle B_\alpha: \alpha \in Y  \rangle$ of sets of cardinality $<\theta,$ there exists $Z \subseteq Y, Z\neq \emptyset$ mod $D$ such that $\langle B_\alpha: \alpha \in Z  \rangle$ forms a $\Delta$-system, i.e., there is $B^*$ such that for all $\alpha \neq \beta,$ both in $Z$, $B_\alpha \cap B_\beta =B^*.$
\end{definition}
The following is essentially due to Erdos and Rado; we will present a proof for completeness.
\begin{lemma}
Suppose $\kappa$ is regular uncountable and $\forall \alpha<\kappa (|\alpha|^{<\theta}<\kappa).$

$(a)$ If $D$ is a normal filter on $\kappa$ and $\{\delta<\kappa: cf(\delta) \geq \theta \}\in D,$ then $D$ has the the $\Delta$-system $\theta$-property.

$(b)$ If $D$ is the filter of co-bounded subsets of $\kappa,$ then $D$ has the $\Delta$-system $\theta$-property.

$(c)$ If $D=\{S \subseteq \kappa: S\cup \{ cf(\delta) < \theta  \}$ contains a club$   \},$ then $D$ has the $\Delta$-system $\theta$-property.
\end{lemma}
\begin{proof}
$(a)$ Let  $Y \subseteq \kappa, Y\neq \emptyset$ mod $D$, and suppose that $\langle B_\alpha: \alpha \in Y  \rangle$ is a sequence of sets of cardinality $<\theta.$ As $|\bigcup_{\alpha\in Y}B_\alpha|\leq \kappa,$ we can assume that all $B_\alpha$'s, $\alpha\in Y,$ are subsets of $\kappa.$ Also as $\{\delta<\kappa: cf(\delta) \geq \theta \}\in D,$ we can assume that $Y \subseteq\{\delta<\kappa: cf(\delta) \geq \theta \}.$ Define the function $g$ on $Y$ by $g(\alpha)=\sup(B_\alpha\cap\alpha).$ Then for all $\alpha\in Y,$ $g(\alpha) < \alpha$ (as $|B_\alpha|<\theta$ and $cf(\alpha) \geq \theta$), so by normality of $D$, we can find $Y_1\subseteq Y, Y_1 \neq \emptyset$ mod $D$, and $\xi<\kappa$ such that for all $\alpha\in Y_1, g(\alpha)=\xi.$ Then
\begin{center}
$\alpha\in Y_1 \Rightarrow B_\alpha \cap \alpha = B_\alpha \cap \xi.$
\end{center}
As there are only $|\xi|^{<\theta} < \kappa$ many subset of  $\xi$ of cardinality $<\theta,$ and since $D$ is normal, there are $Y_2\subseteq Y_1, Y_2 \neq \emptyset$ mod $D$ and a set $B^*$ such that for all $\alpha\in Y_2, B_\alpha \cap \alpha=B_\alpha\cap \xi=B^*.$ Let
\begin{center}
$X=\{\alpha < \kappa: \forall \xi\in Y_2\cap \alpha$ $(\sup(B_\xi)<\alpha)\}.$
\end{center}
$X$ is a club of $\kappa,$ and hence $X\in D$ (as $D$ contains the club filter by its normality). Set $Z=X\cap Y_2.$ Then $Z \subseteq Y, Z \neq \emptyset$ mod $D$, and
$\langle B_\alpha: \alpha \in Z  \rangle$ forms a $\Delta$-system with root $B^*$.

$(b)$ and $(c)$ follow from $(a).$
\end{proof}
The following lemma will be used in the proof of Theorem 5.2.
\begin{lemma}
The $\Delta$-system $\theta$-property is preserved under $\theta$-closed $\theta^+$-c.c. forcing notions.
\end{lemma}
\begin{proof}
Suppose $\MPB$ is a $\theta$-closed $\theta^+$-c.c. forcing notion, $G$ is $\MPB$-generic over $V$ and let $D \in V$ be a filter on $\kappa > \theta$
 which has the  $\Delta$-system $\theta$-property. Let $\tilde D$ be the filter generated by $D$ in $V[G].$ We are going to show that $\tilde D$ has the $\Delta$-system $\theta$-property.

Thus suppose that  $\tilde Y \in V[G], \tilde Y \subseteq \kappa, \tilde Y\neq \emptyset$ mod $\tilde D$, and let $\langle B_\alpha: \alpha \in \tilde Y  \rangle \in V[G]$ be a sequence of sets of cardinality $<\theta$. By the properties of the forcing notion, each $B_\alpha \in V$ and  we can find $Y \in V$ with $Y \subseteq \tilde Y$
and $Y \neq \emptyset$ mod $D$ such that $\langle B_\alpha: \alpha \in Y  \rangle \in V$.

As $D$ has the  $\Delta$-system $\theta$-property,  we can find $Z \subseteq Y, Z\neq \emptyset$ mod $D$ such that $\langle B_\alpha: \alpha \in Z  \rangle$ forms a $\Delta$-system. Then $Z\neq \emptyset$ mod $\tilde D$
and so in $V[G], \tilde D$ has the  $\Delta$-system $\theta$-property, as requested.
\end{proof}

\subsection{On the forcing notion $\PP_\chi$}
Fix $n(*), \theta, \kappa, \chi$ and $D$  as in Theorem 5.2. We describe a cofinality preserving forcing notion $\PP_\chi$ which adds a set $\mathcal{A} \subseteq [\lambda]^{<\aleph_0}$ which has some nice properties.
\begin{definition}
 $\PP_\chi=\langle \langle \PP_i: i \leq \chi \rangle, \langle \lusim{\QQ}_i: i<\chi  \rangle \rangle$ is defined as a $(<\theta)$-support iteration of forcing notions such that:
\begin{enumerate}
\item $(\QQ_0, \leq )$ is defined by:
\begin{enumerate}
\item [(1-1)]$p\in \QQ_0$ iff $p=(w^p, \mathcal{A}^p),$ where $w^p \in [\kappa]^{<\theta}$ and $\mathcal{A} \subseteq [w]^{n(*)}.$
\item [(1-2)] $p \leq q \Leftrightarrow w^q \subseteq w^p$ and $\mathcal{A}^q = \mathcal{A}^p \cap [w^q]^{n(*)}.$
\end{enumerate}
Also let $\lusim{\mathcal{A}}=\bigcup \{\mathcal{A}^p: p\in \dot{G}_{\QQ_0}\}.$
\item Assume $0<i<\chi,$ and $\PP_i$ is defined. Then for some $\PP_i$-names $\lusim{Y}_i$ and $\langle \lusim{w}_i^*, \langle \lusim{w}^i_\alpha: \alpha\in \lusim{Y}_i \rangle \rangle$ we have:
\begin{enumerate}
\item [(2-1)] $\Vdash_{\PP_i}$``$\lusim{Y}_i$ is a subset of $\kappa, \lusim{Y}_i \neq \emptyset$ mod $D$'',
\item [(2-2)]  $\Vdash_{\PP_i}$``$\langle  \lusim{w}^i_\alpha: \alpha\in \lusim{Y}_i \rangle$ is a $\Delta$-system of subsets of $\kappa$ with root $\lusim{w}_i^*$, each of cardinality $\leq \aleph_0$'',
\item [(2-3)] $\Vdash_{\PP_i}$``$\lusim{\QQ}_i=\{u \subseteq \lusim{Y}_i: |u|<\theta$ and if $m<n(*), \alpha_0, \dots, \alpha_{m-1}\in u$ are distinct, then for all $y\in \lusim{\mathcal{A}}, y \subseteq \bigcup_{l<m}\lusim{w}^i_{\alpha_l} \Rightarrow \exists l<m, y\subseteq \lusim{w}^i_{\alpha_l}   \}$''.
\item [(2-4)] $\Vdash_{\PP_i}$``$\leq_{\lusim{\QQ}_i} = \supseteq,$
\end{enumerate}

\item If $\lusim{Y}$ and $\langle \lusim{w}^*, \langle \lusim{w}_\alpha: \alpha\in \lusim{Y} \rangle \rangle$  are $\PP_i$-names of objects as above, then for some $j\in (i,\chi),$ they are of the form $\lusim{Y}_j$ and $\langle \lusim{w}_j^*, \langle \lusim{w}^j_\alpha: \alpha\in \lusim{Y}_j \rangle \rangle$.
\end{enumerate}
\end{definition}
\begin{remark}
$(a)$ $(3)$ can be achieved by a bookkeeping argument, and using the fact that the forcing $\PP_\chi$ satisfies the $\theta^+$-c.c. (see below).

$(b)$ It also follows from the $\theta^+$-chain condition of the forcing that under the same assumptions as $(3), \Vdash_{\PP_j}$``$\lusim{Y} \neq \emptyset$ mod $D$'', so $\QQ_j$ is well-defined.
\end{remark}
\begin{lemma}
Let $\alpha \leq \chi.$

$(a)$ $\PP_\alpha^*$ is a dense subset of $\PP_\alpha,$ where $\PP_\alpha^*$ consists of those $p\in \PP_\alpha$ such that:

$\hspace{0.5cm}$$(1)$ $i\in \dom(p) \Rightarrow p(i)$ is an object (and not just a $\PP_i$-name),

$\hspace{0.5cm}$$(2)$ $0\in \dom(p)$ and for some $w$ we have $w^{p(0)}=w,$

$\hspace{0.5cm}$$(3)$ If $i\in \dom(p),$ then for all $j\in p(i), p \upharpoonright i$ decides $\lusim{w}^*_i, \lusim{w}^i_j.$

$(b)$ If $\alpha < \chi$, then $\Vdash_{\MPB_\alpha}$``$\lusim{\MQB}_\alpha$ is $\theta^+$-Knaster''.

$(c)$ Each $\PP_\alpha$ satisfies the $\theta^+$-c.c.
\end{lemma}
\begin{proof}
$(a)$ follows easily by induction on $\alpha$, and using the fact that $\Vdash_{\PP_i}$``$\lusim{\QQ_i}$ is $\theta$-closed''. Let's present a proof for completeness.
\begin{enumerate}
\item [Case 1.] {\bf $\alpha=0$:} There is noting to prove.

\item [Case 2.] {\bf $\alpha+1$ is a successor ordinal:} Thus assume that $\PP_\alpha^*$ is a dense subset of $\PP_\alpha,$ and let $p\in \PP_{\alpha+1}.$ Then $p \upharpoonright \alpha\in \PP_\alpha,$ so for some $p_1\in \PP_\alpha^*, p_1 \leq_{\PP_\alpha} p\upharpoonright \alpha.$ Since $\Vdash_{\PP_\alpha}$``$\lusim{\QQ_\alpha}$ is $\theta$-closed and $|p(\alpha)|<\theta$'', we can find $q_1$ such that $p_1\Vdash$``$p(\alpha)=q_1$''. As $|q_1|<\theta,$ and again using  $\Vdash_{\PP_\alpha}$``$\lusim{\QQ_\alpha}$ is $\theta$-closed'', we can find $p_2 \leq_{\PP_\alpha} p_1, p_2\in \PP_{\alpha}^*, q_2 \leq_{{\QQ_\alpha}} q_1$ and $w^*_\alpha, w^\alpha_j$ for  $j\in q_2$ such that for all $j\in q_2, p_2\Vdash$``$\lusim{w}^*_\alpha=w^*_\alpha$ and $\lusim{w}^\alpha_j=w^\alpha_j$''. Then $(p_2, q_2)\in \PP^*_{\alpha+1}$ and $(p_2, q_2)\leq_{\PP_{\alpha+1}} p.$

\item [Case 3.] {\bf $\alpha$ is a limit ordinal, $cf(\alpha)\geq \theta$:} Let $p\in \PP_\alpha.$ Then as $|\dom(p)|<\theta,$ we can find $\beta<\alpha$ such that $\dom(p) \subseteq \beta,$ so $p\in \PP_\beta,$ and the induction applies.

\item [Case 4.] {\bf $\alpha$ is a limit ordinal, $cf(\alpha)< \theta$:} Let $\langle \alpha_\xi: \xi<cf(\alpha) \rangle$ be a normal sequence cofinal in $\alpha.$ Let $p\in \PP_\alpha.$ By induction and the $\theta$-closure of forcings, we can find a decreasing sequence $\langle  q_\xi: \xi< cf(\alpha)  \rangle$ of conditions such that:
$q_\xi\in \PP^*_{\alpha_\xi}$ and $q_\xi \leq_{\PP_{\alpha_\xi}} p \upharpoonright \alpha_\xi.$ Let $p_1=\bigcup_{\xi<cf(\alpha)}q_\xi.$ Then $p_1\in \PP^*_\alpha$ and $p_1 \leq_{\PP_\alpha} p.$
\end{enumerate}

$(b)$ can be proved easily by a $\Delta$-system argument.
To prove $(c),$ it suffices, by $(a)$, to show that $\PP_\alpha^*$ satisfies the $\theta^+$-c.c. Let $\{p_\beta: \beta<\theta^+ \} \subseteq \PP_\alpha^*.$ We can assume that:

$(1)$ $\langle  \dom(p_\beta): \beta<\theta^+ \rangle$ forms a $\Delta$-system with root $\Delta,$

$(2)$ For each $i\in \Delta$, $\langle  p_\beta(i): \beta<\theta^+ \rangle$ are pairwise compatible in $\QQ_i$ (using the fact that

$\hspace{0.5 cm}$ $\Vdash_{\PP_i}$``$\lusim{\QQ_i}$ is $\theta^+$-Knaster'').

Now let $\beta_1 < \beta_2 < \theta^+.$ Let $q$ be defined as follows:
\begin{itemize}
\item $\dom(q)=\dom(p_{\beta_1}) \cup \dom(p_{\beta_2}),$
\item $q(0)= \langle w^{p_{\beta_1}} \cup w^{p_{\beta_2}}, \mathcal{A}^{p_{\beta_1}} \cup \mathcal{A}^{p_{\beta_2}}  \rangle,$
\item For all $i \in \dom(q), q(i)=p_{\beta_1}(i) \cup p_{\beta_2}(i)$ (where we assume $p_{\beta_k}(i)=\emptyset$, if $i \notin \dom(p_{\beta_k})$).
\end{itemize}
Clearly $q\in \PP^*_\alpha$, and it extends both $p_{\beta_1}, p_{\beta_2}.$ So  $\{p_\beta: \beta<\theta^+ \}$ is not an antichain.
\end{proof}

\subsection{On the forcing notion $\QQ_{\lambda, \mathcal{A}}$}

In this subsection we describe a forcing notion $\QQ_{\lambda, \mathcal{A}}$, which depends on a parameter $\mathcal{A} \subseteq [\lambda]^{<\aleph_0}.$  For $\mathcal{A} \subseteq [\lambda]^{<\aleph_0},$ set
\begin{center}
$\mathcal{A}^+=\{u\in [\lambda]^{<\aleph_0}: u$ includes some member of $\mathcal{A}    \}.$
\end{center}
\begin{definition}
Assume $\lambda$ is a cardinal and $\mathcal{A} \subseteq [\lambda]^{<\aleph_0}.$ We define the forcing notion $(\QQ_{\lambda, \mathcal{A}}, \leq )$ as follows:

$(a)$ $p\in \QQ_{\lambda, \mathcal{A}}$ iff $p$ is a finite partial function from $\lambda$ to $2^{n(p)},$ for some $n(p)<\omega.$

$(b)$ For $p, q \in \QQ_{\lambda, \mathcal{A}}, p \leq q$ ($p$ is stronger than $q$) iff:

$\hspace{0.5cm}$ $(b-1)$ $\dom(q) \subseteq \dom(p),$

$\hspace{0.5cm}$ $(b-2)$ $\alpha\in \dom(q) \Rightarrow q(\alpha) \unlhd p(\alpha),$

$\hspace{0.5cm}$ $(b-3)$ If $u\in \mathcal{A}$, $u \subseteq \dom(q)$ and $n(q) \leq k < n(p),$ then for some $\alpha\in u, p(\alpha)(k)=0.$
\end{definition}
We also define the following $\QQ_{\lambda, \mathcal{A}}$-names:
\begin{itemize}
\item [$(\alpha)$] $\lusim{\eta}_\alpha =\bigcup \{p(\alpha): \alpha\in \dom(p)$ and $p\in \dot{G}_{\QQ_{\lambda, \mathcal{A}}} \},$
\item [$(\beta)$] $\lusim{a}_\alpha =\{k<\omega: \lusim{\eta}_\alpha (k)=1 \},$
\item [$(\gamma)$] $\lusim{a}_{\alpha, n}=\lusim{a}_{\omega\cdot\alpha+n}.$
\end{itemize}
\begin{remark}
Given any $w \subseteq \lambda,$ let $\mathcal{A}\upharpoonright w =\{u\in \mathcal{A}: u \subseteq w \}.$ Then we  define $\QQ_{\lambda, \mathcal{A}} \upharpoonright w$  to be $\QQ_{\lambda, \mathcal{A} \upharpoonright w}$ which is defined in the natural way. Then for disjoint $w, v \subseteq \lambda$ if
\begin{center}
 $\forall u \in \mathcal{A} (u \subseteq w \cup v \Rightarrow u \subseteq w$ or $u \subseteq v),$
\end{center}
then we have a forcing isomorphism $\QQ_{\lambda, \mathcal{A}}\upharpoonright (w\cup v) \approx  \{(p,q) \in (\QQ_{\lambda, \mathcal{A}}\upharpoonright w) \times (\QQ_{\lambda, \mathcal{A}}\upharpoonright v): n(p)=n(q)\}.$
But  in general the above forcing isomorphism may not be true, if $w,v$ do not satisfy the above requirement, as $(b-3)$ may fail.
\end{remark}
We have the following easy lemma.
\begin{lemma}
Let $\QQ=\QQ_{\lambda, \mathcal{A}}.$ Then

$(a)$ $\QQ$ is a $c.c.c.$ forcing notion,

$(b)$ $\Vdash_{\QQ}$``$\lusim{\eta}_\alpha \in 2^\omega$ and $\lusim{a}_{\alpha, n} \subseteq \omega$'',

$(c)$ $\Vdash_{\QQ}$`` $\bigcap_{i<n} \lusim{a}_{\alpha_i, m}$ is finite'' iff $\{\omega\cdot\alpha_i+m: i<n   \}\in \mathcal{A}^+.$
\end{lemma}
\begin{proof}
$(a)$ follows by a simple $\Delta$-system argument and $(b)$ is clear. Let us prove $(c).$ First assume that  $\{\omega\cdot\alpha_i+m: i<n   \}\in \mathcal{A}^+.$ Then for some $u\in \mathcal{A}, u \subseteq \{\omega\cdot\alpha_i+m: i<n   \},$ and since $\Vdash_{\QQ}$`` $\bigcap_{i<n}\ \lusim{a}_{\alpha_i, m} \subseteq \bigcap_{\alpha\in u} \lusim{a}_{\alpha, m}$'', so we can assume w.l.o.g. that   $\{\omega\cdot\alpha_i+m: i<n   \}\in \mathcal{A}.$ Now let $p\in \QQ.$ By extending $p$, if necessary, we can assume that  $\{\omega.\alpha_i+m: i<n   \} \subseteq \dom(p).$ But then by clause $(b-3),$ any $q\leq p$ forces ``$\bigcap_{i<n} \lusim{a}_{\alpha_i, m} \subseteq n(p)$''. The result follows immediately.

Conversely suppose that
 $\{\omega\cdot\alpha_i+m: i<n   \}\notin \mathcal{A}^+.$ 
Let $p \in \QQ$ and $k<\omega.$ We find $q \leq p$ and $k'> k$ such that $q \Vdash$``$k' \in \bigcap_{i<n} \lusim{a}_{\alpha_i, m}$''.
By extending $p$ we may assume that $\dom(p) \supseteq \{\omega\cdot\alpha_i+m: i<n   \}$
and $n(p) > k.$ Now define $q \leq p$ as follows:
\begin{itemize}
\item $\dom(q)=\dom(p)$.
\item $n(q)=n(p)+1$.
\item If $\alpha \in \dom(p) \setminus \{\omega\cdot\alpha_i+m: i<n   \},$ then $q(\alpha)=p(\alpha)^{\frown}$$\langle (n(p), 0) \rangle$.
\item If $i<n,$ then $q(\omega\cdot\alpha_i+m)=p(\omega\cdot\alpha_i+m)^{\frown}$$\langle (n(p), 1) \rangle$.
\end{itemize}
$q$ is easily seen to be well-defined and clearly
\begin{center}
 $q \Vdash$``$k< k' \in \bigcap_{i<n} \lusim{a}_{\alpha_i, m}$'',
\end{center}
where $k'=n(p).$ Let us show that $q \leq p.$ It suffices to show that it satisfies clause $(b-3)$ of Definition 5.9.
Thus let $u\in \mathcal{A}$ be such that $u \subseteq \dom(p)$. We are going to find some $\alpha \in u$
such that $q(\alpha)(n_p)=0.$ As $\{\omega\cdot\alpha_i+m: i<n   \}\notin \mathcal{A}^+,$  $ u \setminus \{\omega\cdot\alpha_i+m: i<n   \} \neq \emptyset.$ Let $\alpha \in u \setminus \{\omega\cdot\alpha_i+m: i<n   \}$. Then by our definition,
 $q(\alpha)(n_p)=0,$ as requested.
\end{proof}
We now consider the combinatorial principle $C^D_T(\kappa, J)$ in the forcing extensions by $\QQ_{\lambda, \mathcal{A}},$ and  show that the truth or falsity of it depends on the choice of $D$ and $\mathcal{A}.$ For the rest of this subsection, let $J = [\omega]^{<\omega}$, the ideal of bounded subsets of $\omega.$
In the next lemma we discuss conditions on $D$ and $\mathcal{A}$ which imply  $\neg C^D_T(\kappa, J)$ in the forcing extensions by $\QQ_{\lambda, \mathcal{A}}.$
\begin{lemma}
Assume:

$(1)$ $\aleph_0 <\kappa=cf(\kappa) \leq \lambda,$

$(2)$ $D$ is a filter on $\kappa$ with the $\Delta$-system $\aleph_0$-property,

$(3)$ $\mathcal{A} \subseteq [\lambda]^{< \aleph_0}$ and $T$ is a subtree of $\omega^{<\omega},$

$(4)$ There exists some $Y^*\in D^+$ such that;

$\hspace{0.5cm}$$(a)$ If $Y \subseteq Y^*, Y \neq \emptyset$ mod $D$, then there are $t\in T\cap \omega^n$ and distinct $\alpha_0, \dots, \alpha_{n-1}\in Y$

$\hspace{0.9cm}$ such that $\{   \omega.\alpha_i+t(i): i<n \}\in \mathcal{A}.$

$\hspace{0.5cm}$$(b)$ If   $t\in T\cap \omega^n$ and for $i<n,$ $Y_i \subseteq Y^*, Y_i \neq \emptyset$ mod $D$, then there are $\alpha_i\in Y_i,$

$\hspace{0.9cm}$ for $i<n$ such that  $\{   \omega.\alpha_i+t(i): i<n \}\notin \mathcal{A}.$

Then $\Vdash_{\QQ_{\lambda, \mathcal{A}}}$``$\neg C^D_T(\kappa, J)$''.
\end{lemma}
\begin{proof}
Let $\QQ=\QQ_{\lambda, \mathcal{A}},$ and suppose $\nVdash_{\QQ}$``$\neg C^D_T(\kappa, J)$''. By Lemma 5.4, $\Vdash_{\QQ}$``$\langle \lusim{a}_{\alpha, n}: \alpha \in Y^*, n<\omega \rangle$ is a $(\kappa\times \omega)$-matrix for $D$ (i.e., $Y^*\in D^+$)'',  so by our assumption one of the following holds:

{\bf Case 1.} There are $p\in \QQ $ and  $\lusim{X}$ such that:

$\hspace{0.5cm}$$p \Vdash$``$\lusim{X} \subseteq Y^*, \lusim{X} \neq \emptyset$ mod $D$'',

$\hspace{0.5cm}$$p \Vdash$``For every $t\in T\cap \omega^n$ and  distinct $\alpha_0, \dots , \alpha_{n-1}\in \lusim{X}, \bigcap_{i<n}\lusim{a}_{\alpha_i, t(i)} \neq \emptyset$ mod $J$''.

Let $X^*=\{\alpha\in Y^*: p \nVdash$``$\alpha\notin \lusim{X}$''$   \}.$ Then $X^*\in V$ and $p\Vdash$``$\lusim{X} \subseteq X^* \subseteq Y^*$'', so  $X^* \neq \emptyset$ mod $D$. For any $\alpha\in X^*,$ let $p_\alpha \leq p$ be such that $p_\alpha\Vdash$``$\alpha\in \lusim{X}$''. As $D$ has the $\Delta$-system $\aleph_0$-property, we can find $X_1 \subseteq X^*, X_1 \neq \emptyset$ mod $D$ such that $\{\dom(p_\alpha): \alpha\in X_1 \}$ forms a $\Delta$-system with some root, say, $\Delta.$ Let $\Delta=\{ \beta_0, \dots, \beta_{k^{**}-1}\},$ and for each $\alpha\in X_1,$ let
\begin{center}
$\dom(p_\alpha)=\{\beta_{\alpha, j}: j<k_\alpha \},$
\end{center}
where for $j< k^{**}, \beta_{\alpha, j}=\beta_j.$ By shrinking $X_1,$ and using the $\Delta$-system $\aleph_0$-property of $D$, we can further suppose that:
\begin{enumerate}
\item There is some $k^*<\omega$ such that $\alpha\in X_1 \Rightarrow k_\alpha=k^*,$
\item $\alpha, \beta\in X_1 \Rightarrow p_\alpha \upharpoonright \Delta = p_\beta \upharpoonright \Delta,$
\item $\{p_\alpha(\beta_{\alpha, j}): \alpha\in X_1  \}$ is constant, for each $j<k^*.$
\end{enumerate}
Now by $(4-a),$ there are $t\in T\cap \omega^n$ and distinct $\alpha_0, \dots, \alpha_{n-1}\in X_1$ such that $\{   \omega.\alpha_i+t(i): i<n \}\in \mathcal{A}.$ Let $q$ be a common extension of  $p_{\alpha_i}, i<n$, which exists by our above assumptions. Then $q\Vdash$``$\alpha_0, \dots, \alpha_{n-1}\in \lusim{X}$'', and by Lemma 5.10$(c),$
\begin{center}
$q\Vdash$``$ \bigcap_{i<n}\lusim{a}_{\alpha_i, t(i)}$ is finite'',
\end{center}
which is a contradiction.

{\bf Case 2.} There are $p\in \QQ, t\in T\cap \omega^n$ and $\lusim{X}_0, \dots, \lusim{X}_{n-1}$ such that

$\hspace{0.5cm}$$p \Vdash$``$\lusim{X}_i \subseteq Y^*, \lusim{X}_i \neq \emptyset$ mod $D$'' for all $i<n$,

$\hspace{0.5cm}$$p \Vdash$``If $\alpha_i\in \lusim{X}_i$ are distinct, then $\bigcap_{i<n}\lusim{a}_{\alpha_i, t(i)} = \emptyset$ mod $J$''.

For $i<n$ set $X_i^*= \{\alpha\in Y^*: p \nVdash$``$\alpha\notin \lusim{X}_i$''$   \}.$ Then $X_i^*\in V$ and $p\Vdash$``$\lusim{X}_i \subseteq X_i^* \subseteq Y^*$'', so  $X_i^* \neq \emptyset$ mod $D$.
We now proceed by induction on $i<n$ and find $p_{i,\alpha} \leq p$ for $\alpha\in X_i^*$ so that:
\begin{enumerate}
\item $p_{i,\alpha}\Vdash$``$\alpha\in \lusim{X}_i$'',
\item If $\alpha \in X_i^* \cap X_j^*,$ for $i<j<n,$ then $p_{i, \alpha}=p_{j, \alpha}.$
\end{enumerate}
Now proceed as in case 1, and shrink each $X_i^*$ to some $X_{i,1}$ so that:
\begin{enumerate}
\item [(3)]$X_{i,1} \neq \emptyset$ mod $D$,
\item [(4)]$\{\dom(p_{i,\alpha}): \alpha\in X_{i,1} \}$ forms a $\Delta$-system with  root, say, $\Delta_i=\{ \beta_{i,0}, \dots, \beta_{i,k_i^{**}-1}\}.$
\item [(5)]For some $k_i^*<\omega$, $\dom(p_{i, \alpha})=\{\beta_{i,\alpha, j}: j<k_i^* \},$ where for $j< k_i^{**}, \beta_{i,\alpha, j}=\beta_{i,j}.$
\item [(6)]$\alpha, \beta\in X_{i,1} \Rightarrow p_{i,\alpha} \upharpoonright \Delta_i = p_{i,\beta} \upharpoonright \Delta_i,$
\item [(7)]$\{p_{i,\alpha}(\beta_{i,\alpha, j}): \alpha\in X_{i,1}  \}$ is constant, for each $j<k_i^*.$
\end{enumerate}
We again use the $\Delta$-system argument successively on $X_{i, 1}$'s to shrink them to some $X_{i,2}, i<n,$ so that for all $i,j<n$:
\begin{enumerate}
\item [(8)] $X_{i,2}\neq \emptyset$ mod $D$,
\item [(9)] For all $\alpha\in X_{i,2}, \beta\in X_{j,2}, \dom(p_{i,\alpha})\cap \dom(p_{j, \beta})=\Delta_{i,j},$ for some fixed set $\Delta_{i,j},$
\item [(10)] For all $\alpha\in X_{i,2}, \beta\in X_{j,2}, p_{i,\alpha}\upharpoonright \Delta_{i,j}=p_{j,\beta}\upharpoonright \Delta_{i,j}.$
\end{enumerate}
Now by $(4-b),$ there are $\alpha_i\in X_{i,2}, i<n,$ such that $\{   \omega.\alpha_i+t(i): i<n \}\notin \mathcal{A}.$ Let $q$ be a common extension of  $p_{i,\alpha_i}, i<n$, which exists by our above assumptions. Then $q\Vdash$``$\alpha_0, \dots, \alpha_{n-1}\in \lusim{X}$'', and by Lemma 5.10$(c),$
\begin{center}
$q\Vdash$``$ \bigcap_{i<n}\lusim{a}_{\alpha_i, t(i)}$ is infinite'',
\end{center}
which is a contradiction.
The lemma follows.
\end{proof}
we now discuss conditions on $D$ and $\mathcal{A}$ which imply  $C^D_T(\kappa, J)$ in the forcing extensions by $\QQ_{\lambda, \mathcal{A}}.$
\begin{lemma}
Assume:

$(1)$ $D$ is a $\kappa$-complete filter on $\kappa,$ where $\kappa=cf(\kappa) > \aleph_0$ and $\forall \alpha<\kappa (|\alpha|^{\aleph_0} < \kappa),$

$(2)$ $T \subseteq \omega^{<n(*)}$ is a subtree, where $n(*)< \omega,$

$(3)$ $\lambda\geq \kappa$ and $\mathcal{A} \subseteq [\lambda]^{<\aleph_0},$

$(4)$ If $Y \subseteq \kappa, Y \neq \emptyset$ mod $D$, and if $\langle w_\alpha: \alpha\in Y \rangle$ is such that $w_\alpha  \in [\lambda]^{\leq \aleph_0},$ for $\alpha\in Y,$ then there exists $X \subseteq Y, X \neq \emptyset$ mod $D$  such that;

$\hspace{0.5cm}$$(a)$ $\langle w_\alpha: \alpha\in X \rangle$ form a $\Delta$-system with root, say, $w^*$ such that for all $\alpha \neq \beta$ in

$\hspace{0.9cm}$ $X, w_\alpha \cap w_\beta=w^*$, and for all $\gamma\in w^*, otp(w_\alpha\cap \gamma)=otp(w_\beta)\cap \gamma$,

$\hspace{0.5cm}$$(b)$ If $\alpha_0, \dots, \alpha_{n(*)-1}\in X$ are distinct, then
\begin{center}
$\forall u (u\in \mathcal{A}$ and $u \subseteq \bigcup_{i<n(*)}w_{\alpha_i} \Rightarrow \exists i<n(*), u \subseteq w_{\alpha_i}).$
\end{center}
Then $\Vdash_{\QQ_{\lambda, \mathcal{A}}}$``$C^D_T(\kappa, J)$''.
\end{lemma}
\begin{proof}
Let $\QQ=\QQ_{\lambda, \mathcal{A}},$ and suppose $G$ is $\QQ$-generic over $V$. Assume on the contrary that $V[G]\models$``$\neg C^D_T(\kappa, J)$''. Let $p\in G$, $\lusim{X}$ and $\langle \lusim{b}_{\alpha, n}: \alpha\in \lusim{X}, n<\omega \rangle$ be such that

$\hspace{0.5cm}$ $p\Vdash$``$\lusim{X} \subseteq \kappa, \lusim{X} \neq \emptyset$ mod $D$'',

$\hspace{0.5cm}$ $p\Vdash$``$\langle \lusim{b}_{\alpha, n}: \alpha\in \lusim{X}, n<\omega \rangle$ is a counterexample to $C^D_T(\kappa, J)$.

Let $X_1=\{\alpha < \kappa: p \nVdash$``$\alpha\notin \lusim{X}$''$\}.$ Then $X_1\in V$ and $p\Vdash$``$\lusim{X} \subseteq X_1$'', so  $X_1 \neq \emptyset$ mod $D$. For any $\alpha\in X_1,$ let $p_\alpha \leq p$ be such that $p_\alpha\Vdash$``$\alpha\in \lusim{X}$''. We may further assume that $\alpha\in \dom(p).$

For each $\alpha\in X_1,$ we can find $\langle  q_{\alpha, n,m,k}, t_{\alpha, n,m}: n, m, k<\omega  \rangle$ such that:

\begin{enumerate}
\item $t_{\alpha, n,m}:\omega \rightarrow 2,$
\item $\{q_{\alpha, n,m,k}: m<\omega \} \subseteq \QQ$ is a maximal antichain below $p$,
\item $q_{\alpha, n,m,k}\Vdash$``$k\in \lusim{b}_{\alpha, n}$''$ \Leftrightarrow t_{\alpha, n,m}(k)=1.$
\end{enumerate}
We may note that then
\begin{center}
$p\Vdash_{\QQ}$``$\lusim{b}_{\alpha, n}=\{\langle q_{\alpha, n,m,k}, k \rangle: m,k<\omega, t_{\alpha,n,m}(k)=1 \}$'',
\end{center}
and so from now on we assume $\lusim{b}_{\alpha, n}$ is of the form. Let
\begin{center}
$w_\alpha=\dom(p_\alpha) \cup \bigcup \{\dom(q_{\alpha,,n,m,k}): n, m, k<\omega   \} \cup \{\omega.\alpha+n: n<\omega  \}.$
\end{center}
Then each $w_\alpha\in [\lambda]^{\aleph_0}.$ As $D$ is $\kappa$-complete and $\kappa>2^{\aleph_0},$ we can find $X_2, \bar{g}$ and $\bar{t}$ such that

\begin{enumerate}
\item [(4)] $X_2 \subseteq X_1, X_2 \neq \emptyset$ mod $D$,
\item [(5)] $\bar{t}=  \langle  t_{n,m}: n,m<\omega \rangle$ and $\forall \alpha\in X_2, t_{\alpha,n,m}=t_{n,m},$
\item [(6)] For all $\alpha, \beta \in X_2,$ $otp(w_\alpha)=otp(w_\beta)$,
\item [(7)] $\bar{g}= \langle  g_{\alpha,\beta}: \alpha, \beta\in X_2 \rangle,$
\item [(8)] $g_{\alpha, \beta}: w_\beta \cong w_\alpha$ is an order preserving bijection,
\item [(9)] $g_{\alpha,\beta}(\beta)=\alpha,$
\item [(10)] $g_{\alpha,\beta}``[q_{\beta,n,m,k}]=q_{\alpha, n,m,k}.$
\end{enumerate}
Consider $\langle w_\alpha: \alpha\in X_2  \rangle.$ By our assumption, we can find $X_3$ and $w^*$ such that
\begin{enumerate}
\item [(11)] $X_3 \subseteq X_2, X_3 \neq \emptyset$ mod $D$,
\item [(12)] For all $\alpha \neq \beta$ in $X_3, w_\alpha \cap w_\beta=w^*$, and for all $\gamma\in w^*, otp(w_\alpha\cap \gamma)=otp(w_\beta)\cap \gamma$,
\item [(13)] if $\alpha_0, \dots, \alpha_{n(*)-1} \in X_3$ are distinct, then
\begin{center}
$\forall u (u\in \mathcal{A}$ and $u \subseteq \bigcup_{i<n(*)}w_{\alpha_i} \Rightarrow \exists i<n(*), u \subseteq w_{\alpha_i}).$
\end{center}
\end{enumerate}
Note that for $\alpha \neq \beta$ in $X_3, g_{\alpha,\beta} \upharpoonright w^* =id \upharpoonright w^*$ (by $(12)$). As the conclusion of the lemma fails, we can find $q\leq p, q\in G, t\in T$ and $\alpha_0, \dots, \alpha_{n(*)-1}\in X_3$ such that
\begin{center}
$q\Vdash$``$\bigcap_{i<n(*)}\lusim{b}_{\alpha_i, t(i)} = \emptyset$ mod $J$''.
\end{center}
We may suppose that $\dom(q) \subseteq \bigcup_{i<n(*)}w_{\alpha_i}.$\footnote{This is because, by our representation of $\lusim{b}_{\alpha_i, t(i)},$ we can imagine each $\lusim{b}_{\alpha_i, t(i)}$ as a $\QQ\upharpoonright w_{\alpha_i}$-name.}

For $\beta\in X_3$ and $i<n(*)$ set
\begin{center}
$q_{i,\beta}=g_{\beta, \alpha_i}(q \upharpoonright w_{\alpha_i}) \in \QQ \upharpoonright w_\beta.$
\end{center}
Let $\lusim{Y}_i$ be such that
\begin{center}
$\Vdash_{\QQ}$``$\lusim{Y}_i=\{\beta\in X_3: q_{i,\beta} \in \dot{G}_{\QQ}    \}$''.
\end{center}
\begin{claim}
$\Vdash_{\QQ}$``$\lusim{Y}_i \neq \emptyset$ mod $D$''.
\end{claim}
\begin{proof}
Assume not; so there are $r\in G, r\leq q$ and $X\in D^+$ such that $r\Vdash$``$X \cap \lusim{Y}_i=\emptyset$''. As $\dom(r)$ is finite, we can find $\beta\in X$ so that $\dom(r)\cap w_\beta \setminus w^* =\emptyset.$ But then $r, q_{i,\beta}$ are compatible, and any common extension of them forces ``$\beta\in X \cap \lusim{Y}_i$'', which is impossible.
\end{proof}
We show that if $\Vdash_{\QQ}$``$\beta_i\in \lusim{Y}_i$'', for $i<n(*),$ then $\Vdash_{\QQ}$``$\bigcap_{i<n(*)}\lusim{b}_{\beta_i,t(i)} =\emptyset$ mod $J$''.
So assume $V[G]\models$``$\beta_i\in Y_i=\lusim{Y}_i[G]$''. Then $g=\bigcup_{i<n(*)}g_{\alpha_i, \beta_i}$ is an order preserving bijection from $\bigcup_{i<n(*)}w_{\beta_i}$ onto $\bigcup_{i<n(*)}w_{\alpha_i},$ and we can extend it to an automorphism of $\lambda,$ in the natural way, so that its restriction to $\lambda\setminus (\bigcup_{i<n(*)}w_{\beta_i} \cup \bigcup_{i<n(*)}w_{\alpha_i})$ is identity. We denote the resulting function still by $g$. $g$ easily extends to an  automorphism $\hat{g}: \QQ \cong \QQ$ of $\QQ$, which in turn also extends to an automorphism of nice names of $\QQ.$

For $i<n(*), q_{i,\beta_i}\in G_{\QQ},$ so
\begin{center}
$\hat{g}(q_{i,\beta_i})=q \upharpoonright w_{\alpha_i} \in \hat{g}``[G].$
\end{center}
Hence, $q= \bigcup_{i<n(*)}q \upharpoonright w_{\alpha_i} \in \hat{g}``[G].$ But
it is easily seen that $\hat{g}(\lusim{b}_{\beta_i, t(i)})=\lusim{b}_{\alpha_i, t(i)},$ and so
\begin{center}
$\hat{g}^{-1}(q)\Vdash$``$\bigcap_{i<n(*)}\lusim{b}_{\beta_i, t(i)} = \emptyset$ mod $J$''.
\end{center}
On the other hand $\hat{g}^{-1}(q)\in G,$ and the result follows.
\end{proof}
\begin{remark}
Condition $(4-b)$ is implicitly used in the argument to guarantee that the restricted conditions and their union which we defined are well-defined.
See also Remark 5.4.
\end{remark}

\subsection{Proof of Theorem 5.2.} Finally in this subsection we present the proof of theorem 5.2. Thus let $n(*), \theta, \kappa, \chi$ and $D$ be as above and $J=[\omega]^{<\omega}.$ Consider the forcing notion $\PP=\PP_\chi * \lusim{\QQ_{\kappa, \mathcal{A}}},$ where $\mathcal{A} \subseteq [\kappa]^{<\aleph_0}$ is the set added by $\PP_\chi$. It follows from Lemmas 5.7 and 5.10 that $\PP$ is a cofinality preserving forcing notion.

Firs we show that $C^D_n(\kappa, J)$ holds for $n<n(*).$ It suffices to show that in $V^{\PP_\chi},$ the pair $(n(*), D)$ satisfies the the demands in Lemma 5.12. Conditions $(1)-(3)$ from the lemma are clear. To prove $(4-a)$, let $\Vdash_{\PP_\chi}$``$\lusim{Y} \subseteq \kappa, \lusim{Y} \neq \emptyset$ mod $D$,  and  $\langle \lusim{w}_\alpha: \alpha\in \lusim{Y}   \rangle$ is a sequence of countable subsets of $\lambda$''. Let $i<\chi$ be such that $\lusim{Y}$ and $\langle \lusim{w}_\alpha: \alpha\in \lusim{Y}   \rangle$ are $\PP_i$-names. By the fact that in $V^{\MPB_i},$ $D$ has has the $\Delta$-system $\theta$-property (see Lemma 5.5), we can find $\PP_i$-names $\lusim{X}$ and $\lusim{w}^*$ such that:
\begin{enumerate}
\item [(1)] $\Vdash_{\PP_i}$``$\lusim{X} \subseteq \lusim{Y}, \lusim{X} \neq \emptyset$ mod $D$'',
\item [(2)] $\Vdash_{\PP_i}$``$\langle \lusim{w}_\alpha: \alpha\in \lusim{X}   \rangle$ forms a $\Delta$-system with root $\lusim{w^*}$''.
\end{enumerate}
Then for some $j\in (i,\chi),$ $\lusim{X}=\lusim{Y}_j$ and $\langle \lusim{w}^*, \langle \lusim{w}_\alpha: \alpha\in \lusim{Y} \rangle  \rangle =\langle \lusim{w}^*_j, \langle \lusim{w}^j_\alpha: \alpha\in \lusim{Y}_j \rangle  \rangle$. Now by our definition of $\QQ_j,$ we can find $\lusim{Z}\in V^{\PP_{j+1}}$ such that
\begin{enumerate}
\item [(3)] $\Vdash_{\PP_{j+1}}$``$\lusim{Z} \subseteq \lusim{Y}_j, \lusim{Z} \neq \emptyset$ mod $D$'',
\item [(4)] $\Vdash_{\PP_{j+1}}$``If $\alpha_0, \dots \alpha_{n(*)-1}\in \lusim{Z}$ are distinct, then for all $u\in\mathcal{A}, u \subseteq \bigcup_{l<n(*)}\lusim{w}^j_{\alpha_l} \Rightarrow \exists l<n(*), u \subseteq \lusim{w}^j_{\alpha_l}$''.
\end{enumerate}
The result follows immediately, as then the above are also forced to be true by $\PP_\chi.$

Now we show that  $C^D_{n(*)}(\kappa, J)$ fails. Let $T=\omega^{n(*)}.$  We show that
\begin{center}
$\Vdash_{\PP_\chi * \lusim{\QQ}_{\lambda, \mathcal{A}}}$``$ \langle  \lusim{a}_{\alpha, n}: \alpha<\kappa, n<\omega    \rangle$ exemplify $\neg C^D_{n(*)}(\kappa, J)$'',
\end{center}
where the names $\lusim{a}_{\alpha, n}$ are defined just after Definition 5.8. To this end, we check conditions in Lemma 5.11.  Conditions $(1)-(3)$ from the lemma are clear. For $(4-a),$ assume on the contrary that in $V^{\PP_\chi}, Y\neq \emptyset$ mode $D$ is given and $p\in \PP_\chi, p \Vdash$``$\lusim{Y}$ is a counterexample for $(4-a)$''. In $V$, let $X_1=\{\delta <\kappa: p \nVdash$``$\delta \notin \lusim{Y}$''$\}.$ Then $X_1\in V$ and $X_1 \neq \emptyset$ mod $D$. For any $\delta\in X_1,$ let $p_\delta \leq p$ be such that $p_\delta\Vdash$``$\delta\in \lusim{Y}$''.

As in the proof of Lemma 5.7$(b)$, and using the fact that $D$ has the $\Delta$-system property, we can find $X_2 \subseteq X_1, X_2 \neq \emptyset$ mod $D$ such that $\langle  \dom(p_\delta): \delta\in X_2 \rangle$ form a $\Delta$-system with root $\Delta$ and for all $\delta, \gamma\in X_2, p_\delta \upharpoonright \Delta \parallel p_\gamma \upharpoonright \Delta$ ($p_\delta \upharpoonright \Delta$ is compatible with $p_\gamma \upharpoonright \Delta$) Now let $t \in \omega^{n(*)}$ and let $\delta_0, \dots, \delta_{n(*)-1}$ be in $X_2$ such that for each $l,$ $\delta_{l+1}> \sup\{w^{p_{\delta_j}(0)}: j \leq l    \}$. Let $q$ be an extension of all $p_{\delta_l}, l<n(*)$ such that
\begin{center}
$t\in \omega^{n(*)} \Rightarrow \{\omega.\delta_l+t(l): l<n(*)   \}\in \mathcal{A}^{q(0)},$
\end{center}
and
\begin{center}
$u\in \mathcal{A}^{q(0)}, k<n(*), v \subseteq \bigcup\{w^{p_{\delta_l}}: l<n(*), l\neq k    \} \Rightarrow (\exists l)v \subseteq w^{p_{\delta_l}}.$
\end{center}
For example we can set
\begin{center}
$q(0)= \langle \bigcup_{l<n(*)}w^{p_{\delta_l}(0)} \cup \{\omega\cdot \delta_l+t(l): l < n(*)    \}, \bigcup_{l<n(*)}\mathcal{A}^{p_{\delta_l}(0)} \cup \{ \{\omega.\delta_l+t(l)   \}: t\in \omega^{n(*)}     \}   \rangle$.
\end{center}
Then $q \leq p$ and $q\Vdash$``$\lusim{Y}$ can not be a counterexample to $(4-a)$'', a contradiction.

For $(4-b),$ again assume for some $p\in \PP_\chi, t\in T$ and $\lusim{Y}_i, i<n(*),$ we have $p\Vdash$``$t, \langle \lusim{Y}_i: i<n(*)  \rangle$ are counterexample to $(4-b)$''. For each $i<n(*)$, let $X_i^*=\{\delta <\kappa: p \nVdash$``$\delta \notin \lusim{Y}_i$''$\}.$ Then $X_i^*\in V$ and $X^*_{i} \neq \emptyset$ mod $D$. For any $\delta\in X_{i}^*,$ let $p_{i,\delta} \leq p$ be such that $p_{i,\delta} \Vdash$``$\delta\in \lusim{Y}_i$''. Now proceed as in the proof of Lemma 5.11, case 2, to shrink each $X_i^*$ to some $X_{i, 2},$ such that:
\begin{enumerate}
\item [(5)] $X_{i, 2} \neq \emptyset$ mod $D$,
\item [(6)] $\langle  \dom(p_{i,\delta}): \delta\in X_{i,2} \rangle$ form a $\Delta$-system with root $\Delta_i,$
\item [(7)] $\delta, \gamma\in X_{i,2}, p_{i\delta} \upharpoonright \Delta \parallel p_{i,\gamma} \upharpoonright \Delta,$
\item [(8)] For all $\delta\in X_{i,2}, \gamma\in X_{j,2}, \dom(p_{i,\delta})\cap \dom(p_{j, \gamma})=\Delta_{i,j},$ for some fixed set $\Delta_{i,j},$
\item [(9)] For all $\delta\in X_{i,2}, \gamma\in X_{j,2}, p_{i,\delta}\upharpoonright \Delta_{i,j} \parallel p_{j,\gamma}\upharpoonright \Delta_{i,j}.$
\end{enumerate}
Let $\delta_l\in X_{l,2}, l<n(*)$. Let $q$ be an extension of all $p_{l, \delta_l}, l<n(*)$, such that
\begin{center}
$q(0)= \langle \bigcup_{l<n(*)}w^{p_{\delta_l}} \cup \{ \omega \cdot \delta_l+n: n<\omega, l< n(*)    \}, \bigcup_{l<n(*)}\mathcal{A}^{p_{\delta_l}}     \}   \rangle$.
\end{center}
Then $q \leq p$ and
$\{\omega.\delta_l+t(l): l<n(*)   \}\notin \mathcal{A}^q,$
so
\begin{center}
$q\Vdash$``$\{\omega.\delta_l+t(l): l<n(*)   \}\notin \lusim{\mathcal{A}}$''.
 \end{center}
 So $q\Vdash$``$t, \langle \lusim{Y}_i: i<n(*)  \rangle$ can not be  counterexamples to $(4-b)$'', a contradiction.

\section{On $C^s(\kappa)$ v.s. $C(\kappa)$}
In this section, we consider the difference between the combinatorial principles $C^s(\kappa)$ and $C(\kappa),$ and prove the consistency of ``$C(\kappa)$ holds but $C^s_T(\kappa)$ fails for all non-trivial $T$''.

\begin{lemma}
Assume that:

$(1)$ $\kappa=cf(\kappa) > \aleph_0,$

$(2)$ $S^* \subseteq \kappa$ is a stationary subset of $\kappa,$

$(3)$ $\bar{C}=\langle C_\delta: \delta\in S^* \rangle$ is  such that:

$\hspace{0.5 cm}$ $(a)$ Each $C_\delta$ is a club of $\delta,$

$\hspace{0.5 cm}$ $(b)$ For every club $E$ of $\kappa,$ the set $\{\delta\in S^*: \sup(C_\delta\setminus E)=\delta\}$ is not stationary.

$(4)$ $2 \leq n(*) < \omega$

Then there is $\mathcal{A} \subseteq [\kappa]^{n(*)}$ such that:

$(\alpha)$ If $S_l \subseteq S^*$ is stationary for $l<n(*),$ then we can find $\alpha_l, \beta_l\in S_l, $ for $l<n(*)$ such

$\hspace{0.5 cm}$ that $\alpha_0 < \dots < \alpha_{n(*)-1} < \beta_0 < \dots < \beta_{n(*)-1}$ and $\{\alpha_l: l<n(*)  \} \in \mathcal{A}, \{\beta_l: l<$

$\hspace{0.5 cm}$ $n(*)  \} \notin \mathcal{A}$

$(\beta)$ If $Y\subseteq \kappa$ is unbounded, then for some unbounded subset $Z \subseteq Y$ we have $[Z]^{n(*)}\cap \mathcal{A} \in$

$\hspace{0.5 cm}$ $ \{\emptyset, [Z]^{n(*)}  \}.$
\end{lemma}
\begin{proof}
Let
\begin{center}
$\mathcal{A}=\{ \{\alpha_0, \dots, \alpha_{n(*)-1} \}: \alpha_{n(*)-1}\in S^*$ and  $l<n(*)-1 \Rightarrow otp(\alpha_l \cap C_{\alpha_{n(*)-1}})$ is odd$\}$.
\end{center}
Let's show that $\mathcal{A}$ is as required:

($\alpha$) should be clear; let's prove
($\beta$). So assume  $Y\subseteq \kappa$ is unbounded. So there is $Z_1 \subseteq Y$ of size $\kappa$ such that $Z_1 \subseteq S^*$ or $Z_1\cap S^* =\emptyset.$ If $Z_1\cap S^*=\emptyset,$ then obviously $[Z_1]^{n(*)}\cap \mathcal{A}=\emptyset$ and we are done; so assume $Z_1 \subseteq S^*.$ Define the sequence $\langle \alpha_i: i<\kappa \rangle,$ by induction on $i<\kappa,$ such that:
\begin{enumerate}
 \item [(1)] $\alpha_i\in Z_1, \alpha_i > \sup\{\alpha_j: j<i  \},$
 \item [(2)] $\sup(C_{\alpha_i} \cap \bigcup_{j<i}\alpha_j)$   is minimal.
 \end{enumerate}
Let $E=\{\delta<\kappa: \delta=\sup_{j<\delta}\alpha_j$ is a limit ordinal$   \};$ so that $E$ is a club of $\kappa.$
Set
\[
W_1=\{\delta \in E \cap S^*: \delta < \sup(C_{\alpha_\delta} \cap \delta)            \} \subseteq S^*.
\]
Then $W_1$ is a stationary subset of $\kappa,$ as otherwise we can find a club $C \subseteq E$ which is disjoint from $W_1$
and we get a contradiction with $(3-b)$.
It follows from Fodor's lemma that for some $\alpha^*<\kappa,$ the set
$$W_2=\{ \delta\in W_1: \sup(C_{\alpha_\delta} \cap \delta) =\alpha^* < \delta \}$$
 is stationary. Again by Fodor's lemma, there exists $\delta^* < \kappa$ such that the set
$$W_3=\{ \delta\in W_2: \sup(C_{\alpha_\delta} \cap \alpha^*) =\delta^*  \}$$
is stationary.
Let $Z=\{\alpha_\delta: \delta\in W_3  \}.$ Clearly $Z$ is an unbounded subset of $Y$. We show that $[Z]^{n(*)}\cap \mathcal{A} \in  \{\emptyset, [Z]^{n(*)}  \}$.
Thus suppose that $[Z]^{n(*)}\cap \mathcal{A} \neq \emptyset$.
Let $\delta_0 < \dots < \delta_{n(*)-1} \in W_3$. Then $\alpha_{\delta_{n(*)-1}} \in S^*$ and for $l < n(*)-1$ we have

$\hspace{1.5cm}$ $otp(C_{\alpha_{\delta_{n(*)-1}}} \cap \alpha_{\delta_l}) = otp(C_{\alpha_{\delta_{n(*)-1}}} \cap \alpha^*)+ otp((C_{\alpha_{\delta_{n(*)-1}}}\setminus \alpha^*) \cap \alpha_{\delta_l})$

$\hspace{4.6cm}$ $=\delta^*+otp((C_{\alpha_{\delta_{n(*)-1}}}\setminus \delta_{n(*)-1}) \cap \alpha_{\delta_l})$

$\hspace{4.6cm}$ $=\delta^*+1$ (as $C_{\alpha_{\delta_{n(*)-1}}}\setminus \delta_{n(*)-1}) \cap \alpha_{\delta_l}=\{\alpha^*\}$),

which is  odd. So $\{\alpha_0, \dots, \alpha_{n(*)-1} \} \in \mathcal{A},$ as required.
\end{proof}

\begin{remark}
We can replace $(3-b)$ with $(\alpha)$ $\&$ $ (\beta)$, where:

 $(\alpha)$ For every club $E_1$ of $\kappa,$ there exists a club $E_2 \subseteq E_1$ of $\kappa$, such that for every $\delta\in$

$\hspace{0.5 cm}$ $ S^* \cap E_2,$ we have $\delta=\sup\{ \alpha<\delta: otp(C_\delta\cap \alpha)$ is even$  \}=\sup\{ \alpha<\delta: otp(C_\delta\cap \alpha)$

 $\hspace{0.5 cm}$ is odd$  \},$

 $(\beta)$ There is no increasing continuous sequence $\langle \alpha_i: i<\kappa \rangle$ of ordinals $<\kappa$ such that

$\hspace{0.5 cm}$ $C_{\alpha_{2i+1}} \supseteq \{  \alpha_{2j}: j<i\}$ $($ note that this holds if $\sup\{C_\delta: \delta\in S^*  \}<\kappa).$
\end{remark}
\begin{remark}
We can force the existence of such an $S^*$ and $\bar{C}$ by forcing.
\end{remark}
\begin{theorem}
Assume $\kappa=cf(\kappa)>\aleph_0$, $\forall \alpha<\kappa (|\alpha|^{\aleph_0}<\kappa),$ and let $\mathcal{A} \subseteq [\kappa]^{n(*)}$ be as in the conclusion of Lemma 6.1. Then  for any non-trivial tree $T\subseteq \omega^{\leq n(*)}$, we have $V^{\QQ_{\kappa, \mathcal{A}}}\models$``$C_{T, n(*)}(\kappa)+\neg C^s_{T, n(*)}(\kappa)$''.
\end{theorem}
\begin{proof}
That $C^s_{T, n(*)}(\kappa)$ fails in $V^{\QQ_{\kappa, \mathcal{A}}}$ follows from Lemmas 5.11 and 6.1$(\alpha)$. Also,
$C_{T,  n(*)}(\kappa)$ holds in $V^{\QQ_{\kappa, \mathcal{A}}}$ by Lemmas 5.12 and 6.1$(\beta)$.
\end{proof}

The following lemma can be proved similar to the proof of Lemma 6.1.
\begin{lemma}
Let $S^*$ and $\bar{C}$ be as in Lemma 6.1, and assume any $\delta\in S^*$ has uncountable cofinality. Then there is $\mathcal{A} \subseteq [\kappa]^{<\omega}$ such that:

$(\alpha)$ If $n(*)< \omega$ and $S_l \subseteq S^*$ is stationary for $l<n(*),$ then we can find $\alpha_l, \beta_l\in S_l, $ for

$\hspace{0.5 cm}$ $l<n(*)$ such that $\alpha_0 < \dots < \alpha_{n(*)-1} < \beta_0 < \dots < \beta_{n(*)-1}$ and $\{\alpha_l: l<n(*)  \}$

$\hspace{0.5 cm}$ $\in \mathcal{A}, \{\beta_l: l< n(*)  \} \notin \mathcal{A}$

$(\beta)$ If $Y\subseteq \kappa$ is unbounded, then for some unbounded subset $Z \subseteq Y$ we have $[Z]^{<\omega}\cap \mathcal{A} \in$

$\hspace{0.5 cm}$ $ \{\emptyset, [Z]^{<\omega}  \}.$
\end{lemma}
Finally, we have the following; whose proof is the same as the proof of Theorem 6.4, using Lemma 6.5, instead of Lemma 6.1.
\begin{theorem}
Assume $\kappa=cf(\kappa)>\aleph_0$, $\forall \alpha<\kappa (|\alpha|^{\aleph_0}<\kappa),$ and let $\mathcal{A} \subseteq [\kappa]^{<\omega}$ be as in the conclusion of Lemma 6.5. Then   $V^{\QQ_{\kappa, \mathcal{A}}}\models$``$C(\kappa)$+ For any non-trivial tree $T\subseteq \omega^{<\omega},\neg C^s_{T}(\kappa)$''.
\end{theorem}

School of Mathematics, Institute for Research in Fundamental Sciences (IPM), P.O. Box:
19395-5746, Tehran-Iran.

E-mail address: golshani.m@gmail.com

Einstein Institute of Mathematics, The Hebrew University of Jerusalem, Jerusalem,
91904, Israel, and Department of Mathematics, Rutgers University, New Brunswick, NJ
08854, USA.

E-mail address: shelah@math.huji.ac.il

\end{document}